\documentclass[11pt]{amsart}
\usepackage{amsmath,amssymb,amsthm,mathrsfs}
\usepackage[T1]{fontenc}
\usepackage{enumerate}
\usepackage{verbatim}
\usepackage[noadjust]{cite}
\usepackage[dvipsnames]{xcolor}
\usepackage[colorlinks=true, linkcolor=BrickRed, urlcolor=Blue,citecolor=OliveGreen,menucolor=]{hyperref}

\theoremstyle{plain}
\newtheorem{thm}{Theorem}

\newtheorem{prop}[thm]{Proposition}
\theoremstyle{definition}
\newtheorem{defn}[thm]{Definition}
\theoremstyle{remark}
\newtheorem{rem}[thm]{Remark}

\renewcommand{\H}{\mathcal{H}}

\newcommand{\E}{\mathbb E}
\newcommand{\CP}{\mathbb{CP}}
\renewcommand{\phi}{\varphi}

\DeclareMathOperator{\Gr}{Gr}
\DeclareMathOperator{\Grr}{Gr_{res}(\H)}
\DeclareMathOperator{\Grrj}{Gr_{res}^1(\H)}
\DeclareMathOperator{\Grrp}{Gr_{res}^p(\H)}

\DeclareMathOperator{\Tr}{Tr}

\newcommand{\emphh}[1]{{\bf #1}}

\begin{document}
\title[A note on precotangent bundles: the example of Grassmannians]{A note on precotangent bundles:\\ the example of Grassmannians}
\author[T.~Goli\'nski]{Tomasz Goli\'nski}
\address{University of Bia\l ystok\\ Cio\l kowskiego 1M\\15-245 Bia\l ystok\\ Poland}
\email{tomaszg@math.uwb.edu.pl}

\begin{abstract}We prove the existence of the bundle predual to the tangent bundle (called precotangent bundle) for Grassmannians of reflexive Banach spaces and $p$-restricted Grassmannians of the polarized Hilbert space.
\end{abstract}

\maketitle

\tableofcontents

\section{Introduction}
The problem of finding a criterion for a Banach space to be the dual space of some other Banach space (called a predual space) is one of the open problems of the functional analysis. It is related for example to the Sakai theorem for $W^*$-algebras, which states that $W^*$-algebras are exactly those $C^*$-algebras which admit a predual space. From the point of view of infinite dimensional geometry, considering predual spaces is often more natural and convenient than dual spaces. For example, as demonstrated in \cite{OR}, it is not dual of a Banach Lie algebra, but a predual which carries the canonical Poisson structure (albeit under additional conditions). Such a predual space is called then a Banach Lie--Poisson space, see also \cite{Oext,beltita05,Ratiu-grass,Oind,GO-grass,GT-momentum} for more results and applications of this notion. In the same spirit, considering the bundle predual to the given Banach vector bundle makes it easier to study Poisson structures and their relationship with Banach Lie algebroid structures, see \cite{GJ-algebroid}.

The notion of \emph{the precotangent bundle} (i.e. the predual bundle to the tangent bundle) first appeared in the paper \cite{OR} in the context of Banach Lie groups. It was defined as a certain subbundle of the cotangent bundle having the property that duals of the fibers are equal to the tangent spaces. In that paper, it was shown that the precotangent bundle $T_*G$ of a Banach Lie group $G$ exists given that its Banach Lie algebra admits a predual space. In that case $T_*G$ carries the canonical weak symplectic form. It was also demonstrated that this form on the precotangent bundle $T_*G$ (unlike the canonical form on $T^*G$) leads to a well-defined Poisson bracket on $C^\infty(G)$. In general however, it might not define Hamiltonian vector fields for all smooth functions (i.e. $T_*G$ is not a Banach Poisson manifold in the sense of definition given in \cite{OR}). Note that there are also more general definitions of Poisson manifolds in Banach setting, admitting the case when the Poisson bracket is not defined for all smooth functions, see e.g. \cite{pelletier,neeb14,tumpach-bruhat,GRT-poisson-bial}). Using such a definition it is possible to define Poisson brackets for weak symplectic forms (e.g. either on $T_*M$ or $T^*M$).

As far as we know, the problem of the existence of $T_*M$ for other Banach manifolds $M$ has not yet been addressed. However in the paper \cite{GJ-algebroid}, assuming the existence of $T_*M$, the relationship between the Banach Lie algebroid structure on $E\to M$ and the Poisson bracket given by the canonical form on $T_*M$ was established. In this context the problem of existence of the precotangent bundle is a special case of more general problem of existence of the predual bundle to an arbitrary Banach vector bundle.

For the purposes of this paper, we will formulate the following definition:
\begin{defn}
    The precotangent bundle $T_*M$ of a smooth Banach manifold $M$ is a Banach subbundle of $T^*M$ such that
\[ ({T_*}_p M)^* = T_pM. \]
\end{defn}
Note that in general, due to the non-uniqueness of the predual, the cotangent bundle is also not defined uniquely. 

In this paper, we present several examples of Banach manifolds for which the precotangent bundle exists. Obviously, all manifolds modeled on reflexive Banach spaces trivially fall into this category.

Section~\ref{sec:init} is devoted to the basic discussion of the differential structure of the Grassmannian of a Banach space. In Section~\ref{sec:hilbert} we present in detail the simplest example: the Grassmannian of closed subspaces of a Hilbert space. Section~\ref{sec:grres} is devoted to the discussion of the family of $p$-restricted Grassmannians. In Section~\ref{sec:reflex} we prove the existence of the precotangent bundle for the Grassmannian of a reflexive Banach space.

\section{Initial glance: Grassmannian of a Banach space}\label{sec:init}

Let us consider the Grassmannian $\Gr(\E)$ of a Banach space $\E$. It is defined as the set of split subspaces of $\E$, i.e. closed subspaces $F\subset \E$ admitting a complementary subspace $G$. It has a Banach manifold structure, which was described e.g. in \cite{ratiu-mta}. The charts are indexed by pairs of complementary closed subspaces $F$ and $G$, while the modeling spaces are the Banach space of bounded linear maps $L(F,G)$. 

The question of the existence of the precontangent bundle to $\Gr(\E)$ in the most general case is too wide since even the problem of finding sufficient and necessary conditions for the existence of the predual space to $L(F,G)$ is open as far as we know. We will apply some of the known results in Section~\ref{sec:reflex}.

Let us briefly recall the manifold structure on the Grassmannian $\Gr(\E)$ following \cite[3.1.8.G]{ratiu-mta}. For two closed complementary subspaces $F,G\in\Gr(\E)$, we define the following projections
$\pi_F(G)$, $\pi_G(F)$ with respect to the decomposition $\E = F\oplus G$:
\[ \pi_F(G): E \to G\qquad \pi_G(F) : E \to F.\]
Define the following set
\[ \Omega_G = \{ H \in \Gr(\E) \;|\; H\oplus G = \E \}.\]
By $\phi_{F,G}: \Omega_G\to L(F,G)$ we denote a chart associated with the pair $(F,G)$:
\begin{equation}\label{chart_banach}\phi_{F,G}(H) = \pi_F(G)_{|H}\pi_G(F)_{|H}^{-1}.\end{equation}
The motivation behind the formula \eqref{chart_banach} is as follows: the inverse $\phi_{F,G}^{-1}$ of a chart associates to a linear bounded operator its graph in $F\oplus G = \E$, which is always a closed subspace of $\E$. Obviously, considering all possible decompositions of $\E$ into a pair of closed subspaces, we conclude that the sets $\Omega_G$ for all $G\in\Gr(\E)$ cover the Grassmannian $\Gr(\E)$.

The transition map $\psi_{(F',G'),(F,G)}$ 
can be written as
\begin{equation}\label{trans:banach}
\psi_{(F',G'),(F,G)}(A) = \phi_{F',G'}\circ\phi_{F,G}^{-1}(A) =
\pi_{F'}(G')(1+A)\big(\pi_{G'}(F')(1+A)\big)^{-1}
\end{equation}
for $A\in \phi_{F,G}(\Omega_G\cap \Omega_{G'})\subset L(F,G)$. 

By the usual construction, the transition maps for the tangent bundle $T\Gr(\E)$ are 
\begin{equation}\label{tangent_trans}
\Psi_{(F',G'),(F,G)}(A,X) = 
T\psi_{(F',G'),(F,G)}(A)(X),
\end{equation}
where $X\in L(F,G)$. The transition maps $\tilde\Psi_{V,W}$ for the cotangent bundle $T^*\Gr(\H)$ are given by
\begin{equation}\label{trans_cotan}
\tilde\Psi_{(F',G'),(F,G)} = \big( \Psi_{(F',G'),(F,G)}^{-1}\big)^* = \big( \Psi_{(F,G),(F',G')}\big)^*,
\end{equation} 
where $^*$ denotes the dual map.

\section{Grassmannian of the Hilbert space}\label{sec:hilbert}
We will now focus on the simplest case, where $\E$ is a complex separable Hilbert space $\H$. The situation simplifies considerably. First of all, closed subspaces automatically admit a complementary subspace due to the existence of the orthogonal complement. It is also enough to consider charts for a pairs $(V,V^\perp)$ and modeling spaces in the form $L(V,V^\perp)$, where $V$ is a closed subspace of $\H$. Moreover all closed subspaces of $\H$ are reflexive and satisfy the approximation property. The predual space to $L(V,V^\perp)$ is the space of trace-class operators $L^1(V^\perp,V)$. 

It will be useful to associate with the closed subspace $V$ an orthogonal projector $P_V$ onto $V$. From this point of view, $\Gr(\H)$ can be seen as the set of all orthogonal projectors acting on $\H$. 

In order to describe the precotangent bundle $T_*\Gr(\H)$, let us first write the transition maps for $\Gr(\H)$ and $T\Gr(\H)$ presented in the previous section more explicitly using orthogonal projectors borrowing the notation from \cite{GJS-partiso}. By $\phi_V: \Omega_V\to L(V,V^\perp)$ we denote a chart associated to the element of the Grassmannian $V\in\Gr(\H)$ and $V^\perp$:
\begin{equation}\label{chart}\phi_{V}(W) = (P_V)^\perp P_W P_V (P_V P_W P_V)^{-1},\end{equation}
where the chart domain $\Omega_V$ consists of elements of $\Gr(\H)$ such that projection from $V$ onto that element is an isomorphism, or equivalently, their orthogonal complement is complementary to $V$ as a Banach subspace:
\[ \Omega_V = \{ W \in \Gr(\H) \;|\; V\oplus W^\perp = \H \}.\]
The transition map $\psi_{V,W}$ then looks as follows:
\[
\psi_{V,W}(A) = \phi_{V}\circ\phi_W^{-1}(A) =
P_{V^\perp}(1_W + A)\big(P_V(P_W + A)\big)^{-1}
\]
for $A\in \phi_W(\Omega_W\cap\Omega_V)\subset L(W,W^\perp)$. In consequence, using \eqref{tangent_trans}, the transition maps for the tangent bundle $T\Gr(\H)$ assume the form
\begin{multline*}
\Psi_{V,W}(A,X) = 
T\psi_{V,W}(A)(X) = \big(\psi_{V,W}(A),
P_{V^\perp} X \big(P_V(P_W + A)\big)^{-1}\\
- P_{V^\perp}(1_W + A)\big(P_V(P_W + A)\big)^{-1}P_V X \big(P_V(P_W + A)\big)^{-1}\big),
\end{multline*}
where $X\in L(W,W^\perp)$.

Now, the precotangent bundle is obtained by considering a subbundle of $T^*\Gr(\H)$ with fibers modeled on predual spaces to $L(V,V^\perp)$, i.e. $L^1(V^\perp,V)$. 

\begin{thm}
The precotangent bundle $T_*\Gr(\H)$ exists and locally it is defined as
\[ T_*\Gr(\H)_{|\Omega_V} = (T\phi_V)^*(L^1(V^\perp, V)). \]
\end{thm}
\begin{proof}By definition the fibers of $T_*\Gr(\H)_{|\Omega_V}$ are Banach spaces. Thus locally we get a Banach vector bundle over $\Omega_V$ with fibers predual to fibers of $T\Gr(\H)_{|\Omega_V}$.
We need to show that this definition makes sense globally. Let us consider an element of the cotangent bundle $\tilde\mu\in T^*_U\Gr(\H)$ over the subspace $U\in\Omega_V$ with $A=\phi_V(U)$.
In a chart $\phi_V$ the functional $\tilde\mu$ is given by a trace class operator $\mu\in L^1(V^\perp,V)$ in the following way
\[ \langle \tilde\mu ; \tilde v \rangle = \Tr(\mu v) \]
for $v \in L^1(V,V^\perp)$ and $\tilde v = (T_U\phi_V)^{-1} (A,v)\in T_U\Gr(\H)$. In another chart $\phi_W$, such that $U\in\Omega_W$, the element $\tilde\mu$ is represented by an element $\mu'$ given by
\[
(A',\mu') = \tilde\Psi_{W,V}(A,\mu) = \Psi_{V,W}^*(A,\mu),
\]
where $A'=\psi_{W,V}(A)$ and 
\begin{equation}\label{trans-cotan}
\mu' = \big(P_W(P_V+A)\big)^{-1} \mu P_{W^\perp}\big(1-(1_V+A)\big(P_W(P_V+A)\big)^{-1}P_W\big).
\end{equation} 
It is again a trace class operator. Thus the space of trace class operators is preserved by transition maps for the cotangent bundle. In consequence the locally defined  precotangent bundle $T_*\Gr(\H)_{|\Omega_V}$ extends to the global Banach vector bundle $T_*\Gr(\H)$.
\end{proof}

\section{$p$-restricted Grassmannians}\label{sec:grres}
Consider now a polarized Hilbert space, i.e. a separable complex Hilbert space with a fixed orthogonal decomposition
\[ \H = \H_- \oplus \H_+, \]
where $\H_\pm$ are infinite-dimensional closed subspaces orthogonal to each other. We will denote with $P_\pm$ the orthogonal projectors on $\H_\pm$.

For a parameter $1\leq p < \infty$ we denote by $L^p(\H)$ the $p$-Schatten ideal. By $L^0(\H)$ we denote the ideal of compact operators. It is known that these spaces are reflexive for $p>1$. For $p=1$, the predual space of $L^1(\H)$ is the space $L^0(\H)$, while the dual is $L(\H)$. On the other hand, by Sakai theorem, $L^0(\H)$ does not possess a predual since it is not weakly closed in $L(\H)$. All duality pairings mentioned here are given by the trace.

The \emphh{$p$-restricted Grassmannian} $\Grrp$ is defined as the set of closed subspaces $W\subset\H$ such that:
\begin{enumerate}[i)]
\item the orthogonal projection $p_+:W\to \H_+$ is a Fredholm operator;
\item the orthogonal projection $p_-:W\to \H_-$ is $L^p$ class.
\end{enumerate}


One can equivalently describe $\Grrp$ in the following way:
\[W\in\Grr \Longleftrightarrow P_W-P_+ \in L^p,\]
see \cite{spera-valli,wurzbacher}.

It is known that $\Grrp$ is a manifold modeled on the Banach spaces of $L^p(V,V^\perp)$ with charts given by formula \eqref{chart} (but considered as taking values in $L^p$ spaces), see e.g. \cite{segal,wurzbacher,GJS-partiso} for the case $p=2$ or \cite{segal-wilson} for $p=0$. 

Naturally, due to reflexivity, for $p>1$ the precotangent bundle of $\Grrp$ is equal to the cotangent bundle. For $p=0$ the precotangent bundle does not exist, since the space of compact operators doesn't have a predual. In the case $p=1$, following the same line of reasoning as in the previous section, we get the following result
\begin{thm}
The precotangent bundle $T_*\Grrj$ exists and locally it is defined as
\[ T_*\Grrj_{|\Omega_V} = (T\phi_V)^*(L^0(V^\perp, V)). \]
\end{thm}
\begin{proof}
Analogously to the previous proof, we need to make sure that the transition maps for the cotangent bundle $T^*\Grrj$ preserve the class of compact operators. Transition maps are given by the formula \eqref{trans-cotan}, but this time with $\mu\in L^0(V^\perp,V)$.
From the properties of compact operators, it is immediate that $\mu'$ is also compact. Thus the local definition of $T_*\Grrj$ is correct.
\end{proof}

\section{Grassmannian of a reflexive Banach space}\label{sec:reflex}

This case generalizes the results of Section~\ref{sec:hilbert} to the case of an arbitrary reflexive Banach space $\E$.

\begin{prop}\label{prop:reflexive_predual}
Let $\E$ be reflexive. 
Then for any two closed subspaces $F,G\subset \E$, the space of bounded linear maps $L(F,G)$ admits a predual, which is $F \hat\otimes_\pi G_*$.
\end{prop}
\begin{proof}
Every closed subspace of a reflexive space is also reflexive, see e.g. \cite[1.11.16 Theorem]{megginson}. Thus $F_*=F^*$ and $G_*=G^*$. 

Moreover, the following identification holds
\[ L(X,Y^*) = (X \hat\otimes_\pi Y)^*,\]
where $X\hat\otimes_\pi Y$ denotes the projective tensor product of Banach spaces $X,Y$, see \cite[Section 2.2]{ryan}. On the elements of the algebraic tensor product, the duality pairing is given as follows
\begin{equation}\label{pair} \big\langle T, \textstyle
\sum x_i\otimes y_i \big\rangle = \textstyle\sum \langle Tx_i, y_i \rangle\end{equation}
for $T\in L(X,Y^*)$ and finitely many $x_i\in X$, $y_i\in Y$.

Specializing to our case, we obtain immediately:
\[(L(F,G))_* = F \hat\otimes_\pi G_*.\]
\end{proof}

Moreover if either $F$ or $G_*$ additionally satisfies the approximation property, it is known that $(F \hat\otimes_\pi G_*)$ is the space of nuclear operators $\mathcal N(G,F)$, see \cite[Corrolary 4.8]{ryan}. However assuming the approximation property for $\E$ does not imply that its subspaces (or their duals) satisfy it as well.

In general the projective tensor product of reflexive Banach spaces might not be reflexive itself. The easiest example is a projective tensor product of two copies of a Hilbert space $\H$ which is the space $L(\H,\H^*)$. On the other hand, it is known for example that if every element of $L(F,G)$ is compact, then $L(F,G)$ (and thus $F \hat\otimes_\pi G_*$ as well) is reflexive, see \cite[Theorem 4.19]{ryan}.

\begin{thm}
For the Grassmannian $\Gr(\E)$ of the reflexive Banach space $\E$, the precotangent bundle $T_*\Gr(\E)$ exists.
\end{thm}
\begin{proof}
Analogously to previous sections, by virtue of Proposition~\ref{prop:reflexive_predual}, we define the precotangent bundle locally using the charts \eqref{chart_banach} as follows
\[ T_*\Gr(\E)_{|\Omega_{G}} = (T\phi_{F,G})^*(F \hat\otimes_\pi G_*). \]
It is now sufficient to prove that the transition maps $\tilde\Psi_{(F',G'),(F,G)}$ for $T^*\Gr(\E)$ preserve the projective tensor product $F \hat\otimes_\pi G_*$. As stated in \eqref{trans_cotan} they are dual maps of (the inverse of) transition maps $\Psi_{(F,G),(F',G')}$ for $T\Gr(\E)$. It is straightforward to derive an explicit expression for transition maps $\Psi_{(F,G),(F',G')}$ from \eqref{trans:banach}. We will not however write an explicit formula here as it is sufficient to note only that it is of the form
\[ \Psi_{(F,G),(F',G')}(A,X) = \big(\psi_{(F,G),(F',G')}A, \sum\limits_i T_i X S_i\big),
\]
where $X\in L(F',G')$, the sum is finite and $T_i\in L(G',G)$, $S_i\in L(F,F')$ (they may depend on $A$). In other words, the transition maps at each fiber are given by a combination of left and right multiplication by bounded operators. 

From \eqref{trans_cotan} we have that the transition maps for $T^*\Gr(\E)$ are given as dual maps
\[\tilde\Psi_{(F',G'),(F,G)}(A,\mu) = \big(\psi_{(F,G),(F',G')}A, \mu')\]
for 
\[\mu' = \big(\sum\limits_j T_j \,\cdot\, S_j\big)^*\mu.\]
From the definition of the dual map and the formula for duality pairing \eqref{pair} applied to the elements of the algebraic tensor product $F\otimes G_*$, we have
\begin{multline*}
\big\langle X, \big(\sum\limits_j T_j \,\cdot\, S_j\big)^*\textstyle\sum\limits_i x_i\otimes y_i \big\rangle = \big\langle \sum\limits_j T_j X S_j,\textstyle\sum\limits_i x_i\otimes y_i \big\rangle\\ = 
\textstyle\sum\limits_{i,j} \big\langle T_j X S_j x_i , y_i \big\rangle
=\textstyle\sum\limits_{i,j} \big\langle X S_j x_i , T_j^* y_i \big\rangle=
\big\langle X, \sum\limits_{i,j} S_j x_i \otimes T_j^* y_i \big\rangle
\end{multline*}
for $x_i\in F$, $y_i\in G_*$.
Since algebraic tensor product is dense in projective tensor product, we get that the dual map $\Psi_{(F,G),(F',G')}$ on the predual space is a tensor product of operators
\[\mu' = \big(\textstyle\sum\limits_i S_i\otimes T_i^*\big)(\mu),\]
see also \cite[Proposition 2.3]{ryan}. Naturally the maps of the type $S_i\otimes T_i^*$ map the space $F \hat\otimes_\pi G_*$ to $F' \hat\otimes_\pi G'_*$. In this manner we obtain a Banach bundle $T_*\Gr(\E)$ as a subbundle of $T^*\Gr(\E)$.
\end{proof}

\begin{rem}
The manifold $\Gr(\E)$ is disconnected. The sets $\Gr_k(\E)$ and $\Gr^k(\E)$ of respectively $k$-dimensional and $k$-codimensional subspaces are connected submanifolds, see \cite{ratiu-mta}. Thus the precotangent bundle exists also for $\Gr_k(\E)$ and $\Gr^k(\E)$, including in particular the projective space $\CP(\E)$ of a Banach space $\E$, i.e. the space of one-dimensional subspaces of $\E$.
\end{rem}


\begin{thebibliography}{AMR88}
\providecommand{\url}[1]{\texttt{#1}}
\providecommand{\urlprefix}{URL }

\bibitem[AMR88]{ratiu-mta}
\textsc{Abraham, R.}, \textsc{Marsden, J.E.}, \textsc{Ratiu, T.S.}
\newblock \emph{Manifolds, Tensor Analysis, and Applications}, volume~75 of
  \emph{Applied Mathematical Sciences}.
\newblock Springer-Verlag, Berlin--Heidelberg, second edition, 1988.

\bibitem[BR05]{beltita05}
\textsc{Belti\c{t}\u{a}, D.}, \textsc{Ratiu, T.S.}
\newblock \emph{Symplectic leaves in real {B}anach {L}ie--{P}oisson spaces}.
\newblock Geom. Funct. Anal., \textbf{15}(4):753--779, 2005.
\newblock \urlprefix\url{http://dx.doi.org/10.1007/s00039-005-0524-9}.

\bibitem[BRT07]{Ratiu-grass}
\textsc{Belti\c{t}\u{a}, D.}, \textsc{Ratiu, T.S.}, \textsc{Tumpach, A.B.}
\newblock \emph{The restricted {G}rassmannian, {B}anach {L}ie--{P}oisson
  spaces, and coadjoint orbits}.
\newblock J. Funct. Anal., \textbf{247}:138--168, 2007.
\newblock \urlprefix\url{http://dx.doi.org/10.1016/j.jfa.2007.03.001}.

\bibitem[CP12]{pelletier}
\textsc{Cabau, P.}, \textsc{Pelletier, F.}
\newblock \emph{Almost {L}ie structures on an anchored {B}anach bundle}.
\newblock J. Geom. Phys., \textbf{62}(11):2147--2169, 2012.
\newblock \urlprefix\url{http://dx.doi.org/10.1016/j.geomphys.2012.06.005}.

\bibitem[GJ25]{GJ-algebroid}
\textsc{Goli\'nski, T.}, \textsc{Jakimowicz, G.}
\newblock \emph{Poisson structure on predual of {B}anach {L}ie algebroid}.
\newblock arXiv, 2025.
\newblock \urlprefix\url{https://arxiv.org/abs/2505.13351}.

\bibitem[GJS25]{GJS-partiso}
\textsc{Goli\'nski, T.}, \textsc{Jakimowicz, G.}, \textsc{Sliżewska, A.}
\newblock \emph{Banach {L}ie groupoid of partial isometries over restricted
  {G}rassmannian}.
\newblock Anal. Math. Phys., \textbf{15}(27), 2025.
\newblock \urlprefix\url{http://dx.doi.org/10.1007/s13324-025-01028-y}.

\bibitem[GO10]{GO-grass}
\textsc{Goliński, T.}, \textsc{Odzijewicz, A.}
\newblock \emph{Hierarchy of {H}amilton equations on {B}anach {L}ie--{P}oisson
  spaces related to restricted {G}rassmannian}.
\newblock J. Funct. Anal., \textbf{258}:3266--3294, 2010.
\newblock \urlprefix\url{http://dx.doi.org/10.1016/j.jfa.2010.01.019}.

\bibitem[GRT25]{GRT-poisson-bial}
\textsc{Goliński, T.}, \textsc{Rahangdale, P.}, \textsc{Tumpach, A.B.}
\newblock \emph{{P}oisson structures in the {B}anach setting: comparison of
  different approaches}.
\newblock In P.~Kielanowski, A.~Dobrogowska, D.J. Fernández~C., T.~Goliński,
  eds., \emph{XLI Workshop on Geometric Methods in Physics}, Trends in
  Mathematics, pages 97--118. Birkh{\"a}user, Cham, 2025.
\newblock \urlprefix\url{http://dx.doi.org/10.1007/978-3-031-89857-0_9}.

\bibitem[GT24]{GT-momentum}
\textsc{Goliński, T.}, \textsc{Tumpach, A.B.}
\newblock \emph{Geometry of integrable systems related to the restricted
  {G}rassmannian}.
\newblock SIGMA, \textbf{20}, 2024.
\newblock \urlprefix\url{http://dx.doi.org/10.3842/SIGMA.2024.104}.

\bibitem[Meg98]{megginson}
\textsc{Megginson, R.E.}
\newblock \emph{An introduction to {B}anach space theory}, volume 183.
\newblock Springer, 1998.

\bibitem[NST14]{neeb14}
\textsc{Neeb, K.H.}, \textsc{Sahlmann, H.}, \textsc{Thiemann, T.}
\newblock \emph{Weak {P}oisson structures on infinite dimensional manifolds and
  hamiltonian actions}.
\newblock In \emph{Lie Theory and Its Applications in Physics}, pages 105--135.
  Springer, 2014.
\newblock \urlprefix\url{http://dx.doi.org/10.1007/978-4-431-55285-7_8}.

\bibitem[OR03]{OR}
\textsc{Odzijewicz, A.}, \textsc{Ratiu, T.S.}
\newblock \emph{Banach {L}ie--{P}oisson spaces and reduction}.
\newblock Comm. Math. Phys., \textbf{243}:1--54, 2003.
\newblock \urlprefix\url{http://dx.doi.org/10.1007/s00220-003-0948-8}.

\bibitem[OR04]{Oext}
\textsc{Odzijewicz, A.}, \textsc{Ratiu, T.S.}
\newblock \emph{Extensions of {B}anach {L}ie--{P}oisson spaces}.
\newblock J. Funct. Anal., \textbf{217}:103--125, 2004.
\newblock \urlprefix\url{http://dx.doi.org/10.1016/j.jfa.2004.02.012}.

\bibitem[OR08]{Oind}
\textsc{Odzijewicz, A.}, \textsc{Ratiu, T.S.}
\newblock \emph{Induced and coinduced {B}anach {L}ie--{P}oisson spaces and
  integrability}.
\newblock J. Funct. Anal., \textbf{255}:1225--1272, 2008.
\newblock \urlprefix\url{http://dx.doi.org/10.1016/j.jfa.2008.06.001}.

\bibitem[PS86]{segal}
\textsc{Pressley, A.}, \textsc{Segal, G.B.}
\newblock \emph{Loop Groups}.
\newblock Oxford Mathematical Monographs. Clarendon Press, Oxford, 1986.

\bibitem[Rya02]{ryan}
\textsc{Ryan, R.A.}
\newblock \emph{Introduction to Tensor Products of {B}anach Spaces}.
\newblock Springer, London, 2002.

\bibitem[SV94]{spera-valli}
\textsc{Spera, M.}, \textsc{Valli, G.}
\newblock \emph{Pl{\"u}cker embedding of the {H}ilbert space {G}rassmannian and
  the {CAR} algebra}.
\newblock Russ. J. Math. Phys., \textbf{2}:383--392, 1994.

\bibitem[SW85]{segal-wilson}
\textsc{Segal, G.B.}, \textsc{Wilson, G.}
\newblock \emph{Loop groups and equations of {KdV} type}.
\newblock Inst. Hautes \'Etudes Sci. Publ. Math., \textbf{61}:5--65, 1985.
\newblock \urlprefix\url{http://www.numdam.org/item/?id=PMIHES_1985__61__5_0}.

\bibitem[Tum20]{tumpach-bruhat}
\textsc{Tumpach, A.B.}
\newblock \emph{Banach {P}oisson--{L}ie groups and {B}ruhat--{P}oisson
  structure of the restricted {G}rassmannian}.
\newblock Comm. Math. Phys., \textbf{373}(3):795--858, 2020.
\newblock \urlprefix\url{http://dx.doi.org/10.1007/s00220-019-03674-3}.

\bibitem[Wur01]{wurzbacher}
\textsc{Wurzbacher, T.}
\newblock \emph{Fermionic second quantization and the geometry of the
  restricted {G}rassmannian}.
\newblock In A.~Huckleberry, T.~Wurzbacher, eds., \emph{Infinite Dimensional
  {K\"ahler} Manifolds, DMV Seminar}, volume~31, pages 287--375.
  Birkh{\"a}user, Basel, 2001.
\newblock \urlprefix\url{http://dx.doi.org/10.1007/978-3-0348-8227-9_6}.

\end{thebibliography}

\end{document}